\documentclass[12pt, twoside, leqno]{article}

\usepackage{amssymb,amsthm,amsfonts,verbatim,graphicx,mathtools,mathrsfs,commath,float}
\usepackage{enumerate,setspace}
\usepackage[T1]{fontenc}
\usepackage[hidelinks]{hyperref}

\makeatletter
\@namedef{subjclassname@2010}{%
  \textup{2010} Mathematics Subject Classification}
\makeatother

\newtheorem{ut}{Theorem}

\newtheorem{up}{Proposition}

\newtheorem{uq}{Question}

\theoremstyle{definition}

\newtheorem*{ue}{Example $\mathbf{X_1}$}
\newtheorem*{ue2}{Example $\mathbf{X_2}$}
\newtheorem*{ue3}{Example $\mathbf{X_3}$}
\newtheorem*{ur}{Remark}

\DeclareMathOperator{\Gr}{Gr}

\numberwithin{equation}{section}
\date{}

\def\bysame{\leavevmode\hbox to3em{\hrulefill}\thinspace}

\begin{document}

\baselineskip=17pt

\title{\Large Widely-connected sets in the
 bucket-handle  continuum}

\author{David Lipham\\
Department of Mathematics and Statistics\\ 
Auburn University\\
Auburn, AL 36849, United States\\
E-mail: dsl0003@auburn.edu}

\maketitle

\renewcommand{\thefootnote}{}



\renewcommand{\thefootnote}{\arabic{footnote}}
\setcounter{footnote}{0}

\vspace{-15mm}

\begin{abstract}A connected topological space is said to be widely-connected if each of its non-degenerate connected subsets is dense in the entire space. The object of this paper is the construction of widely-connected subsets of the plane.   We give a completely metrizable example that answers a question of Paul Erd\H{o}s and Howard Cook, while a similar example answers a question of Jerzy Mioduszewski.   
\end{abstract}

\maketitle

\section{Introduction}

A topological space $X$ is \textit{punctiform} if $X$ has no compact connected subset of cardinality greater than one. To avoid the trivial examples, assume from now on that every space has more than one point.  

The earliest examples of connected punctiform spaces are due to Mazurkiewicz \cite{sur}, Kuratowski and Sierpinski  \cite{ks}, and Knaster and Kuratowski \cite{knas}. They exploit the fact that a discontinuous function can have a connected graph.    For instance, let $\varphi(x)=\sin(1/x)$ for $x\neq 0$ and put $\varphi(0)=0$. Then $\text{disc}(\varphi)=\{ 0\}$, but $\Gr(\varphi)$ is connected as it consists of two rays with a common limit point at the origin.\footnote{For any space $X$ and real-valued function $f:X\to \mathbb R$,  we let $\text{disc}(f):=\{x\in X:f\text{ is not continuous at }x\}$ denote the set of discontinuities of $f$. $\Gr(f):=\{\langle x,f(x)\rangle\in X\times \mathbb R:x\in X\}$ denotes the graph of $f$.} Now let $\mathbb Q =\{q_n:n<\omega\}$ be an enumeration of the rationals, and define $f:\mathbb R \to \mathbb R$  by $f(x)=\sum_{n=0}^\infty \varphi(x-q_n)\cdot 2^{-n}.$ The function $f$ satisfies the conclusion of the Intermediate Value Theorem, and $\text{disc}(f)=\mathbb Q$. These properties imply  $\Gr(f)$ is both punctiform and connected.  Further, $\text{Gr} (f)$ is a $G_\delta$-subset of the plane, and is therefore completely metrizable.

There are two   types of punctiform connected spaces that exclude graphs of  functions from the real line. A  connected space $X$ is  \textit{biconnected} if $X$ is not the union of two disjoint connected sets each containing more than one point, and \textit{widely-connected} if every  non-degenerate connected subset of $X$ is dense in $X$.  That  biconnected  and widely-connected Hausdorff spaces are punctiform is an easy  consequence of the Boundary Bumping Theorem,  Lemma 6.1.25 in \cite{eng2}.  

Knaster and Kuratowski  \cite{kk} gave the first  biconnected example by constructing a connected set with a dispersion point. A point $p$ in a connected space $X$ is a \textit{dispersion point} if $X\setminus \{p\}$ is  \textit{hereditarily disconnected}, i.e.,  if no connected subset of $X\setminus \{p\}$ has more than one point. Completely metrizable dispersion point sets are given in \cite{knas} and \cite{rob}. Assuming the Continuum Hypothesis, E.W. Miller \cite{mil}  showed that there is a biconnected set \textit{without} a  dispersion point. Miller's biconnected set is also widely-connected. The original widely-connected set slightly predates Miller's example, and is due to P.M. Swingle \cite{swi}. Neither example is completely metrizable. 

\begin{figure}[ht]
  \centering
  \includegraphics[scale=.57]{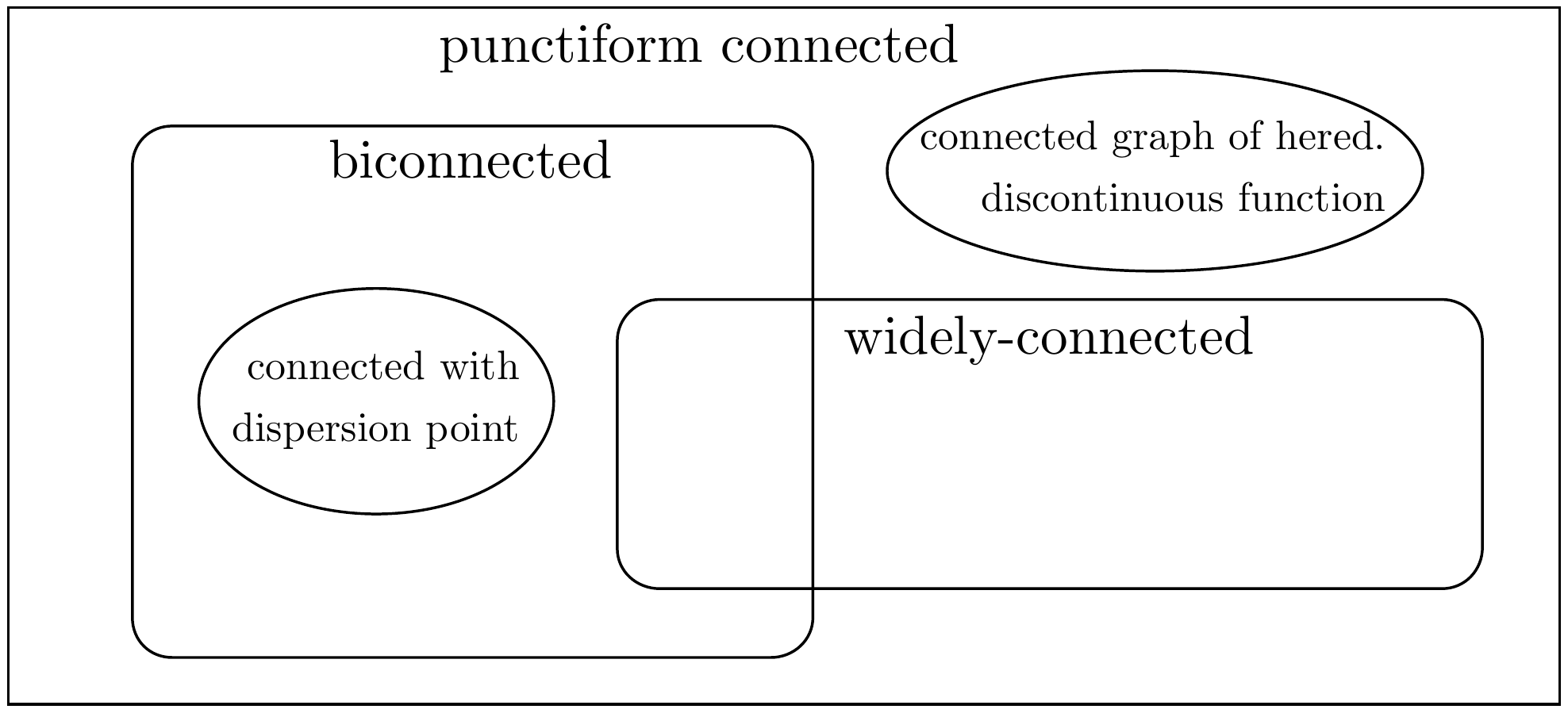} 
   \caption{Classes of some T$_2$ connected  spaces}
\end{figure}

In this paper we describe a method for producing widely-connected subsets of the bucket-handle continuum $K\subseteq \mathbb R ^2$.  An example is obtained in Section \ref{ec} by deleting a countable infinity of nowhere dense compact sets from $K$. Thus, the problem of finding a completely metrizable widely-connected space, due to Paul Erd\H{o}s and Howard Cook (Problem G3 in \cite{rud3} and \cite{erdosproblem}, and Problem 123 in \cite{cookproblem}), is solved.  The example is not biconnected; we will leave open the question of whether there is a completely metrizable biconnected space without a dispersion point. 

Our method produces another example that answers the following question of Jerzy Mioduszewski: \textit{Does every separable metric widely-connected $W$ have a metric compactification $\gamma W$ such that  every composant of $\gamma W$ intersects $W$ in a finite set?}   This  appears as part of Problem 23 in \cite{miodproblem}. The answer is no, even if the word ``finite'' is replaced with ``countable''.

%In this paper we describe a new method for producing widely-connected subsets of a planar indecomposable continuum known as the bucket-handle continuum.  An example is obtained in Section \ref{metex} by deleting only countably many  compact sets from the bucket-handle. A metrizable space is completely metrizable if and only if it is a $G_\delta$-subset of one (each) of its compactifications. Thus, the problem of finding a completely metrizable widely-connected space, due to Paul Erd\H{o}s and Howard Cook, is solved. 

%\newpage

\section{Assembly of connectible sets}

Let $C$ denote the middle-thirds Cantor set in the interval $[0, 1]$. If $X$ is a subset of $C\times (0,1)$, then $X$ is \textit{connectible} if $\langle c,0\rangle\in A$ whenever $A$ is a clopen subset  of $X\cup (C\times \{0\})$, $c\in C$, and  $A\cap (\{c\}\times (0,1))\neq\varnothing$. There is a natural correlation between connectible subsets of $C\times (0,1)$ and connected subsets of the Cantor fan -- this is the intuition behind the examples in Section 3.  In this section we describe a method for constructing widely-connected spaces from certain connectible sets. 
 
Let $K$ be the Knaster bucket-handle continuum. The rectilinear version of $K$ depicted  in Figure 2 is formed by successively removing open middle-third canals (white regions of the figure), beginning with the unit square $[0,1]^2$.  
 $K\setminus \{\langle 0,0\rangle\}$ is locally homeomorphic to $C\times (0,1)$.  Moreover, if  $\Delta=\{\langle x,x\rangle\in K:x\in[0,1]\}$ is the diagonal Cantor set in $K$, then $K\setminus \Delta$ is the union of $\omega$-many  copies of $C\times (0,1)$, to wit, $K\setminus \Delta=\bigcup _{i<\omega}K_i$ where the $K_i$'s are shown in Figure 3.
 
\begin{figure}[H]
\centering
\hspace{-12mm}
\begin{minipage}{.45\textwidth}
  \centering
   \includegraphics[scale=0.65,trim={0 0 16.1cm 0},clip]{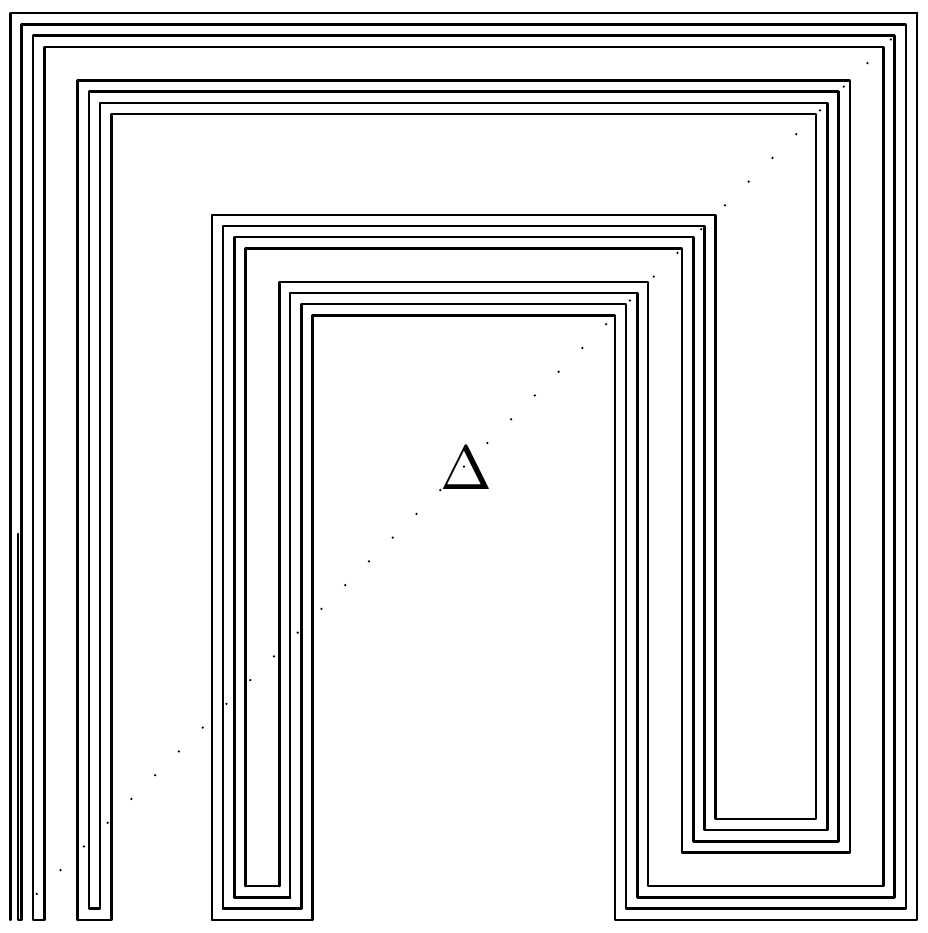} 
  \caption{$K$}
  \label{fig:test1}
\end{minipage}%
\hspace{2mm}
\begin{minipage}{.5\textwidth}
  \centering
  \vspace{.2mm}
  \includegraphics[scale=0.65,trim={12.8cm 0 0 0},clip]{Fig2.pdf} 
  \caption{$K\setminus \Delta$}
  \label{fig:test2}
\end{minipage}
\end{figure}

For each $i<\omega$ let $\varphi_i:C\times(0,1)\to K_i$ be a simple bending homeomorphism that witnesses $K_i\simeq C\times (0,1)$.  If $X$ is a subset of  $C\times (0,1)$ then define$$W[X]=\Delta\cup \bigcup_{i<\omega} \varphi_i [X].$$

\begin{up}\label{one}If $X$ is a dense connectible subset of $C\times (0,1)$, then   $W[X]$ is connected. If, additionally, $X$ is hereditarily disconnected, then $W[X]$ is widely-connected.\footnote{The converse is also true: If $W[X]$ is widely-connected, then $X$ is a dense, connectible, hereditarily disconnected subset of $C\times (0,1)$.}\end{up}

\begin{proof}Let $X$ be a dense connectible subset of $C\times (0,1)$.

Suppose for a contradiction that  $W:=W[X]$ is not connected.  Let  $A$ and $B$  be non-empty closed subsets of $W$ such that $W=A\cup B$ and $A\cap B=\varnothing$. Then $K=\overline A \cup \overline B$ 
because $W$ is dense in $K$. There exists $p\in \overline A \cap \overline B$ because $K$ is connected. $p\notin \Delta$ because $\Delta\subseteq W$, so there exists $i<\omega$ such that $p\in K_i$. Thinking of $\overline{K_i}$ as $C\times [0,1]$ with $C\times \{0,1\}\subseteq \Delta$, give local coordinates $\langle p(0),p(1)\rangle$ to $p$. Assume that $p':=\langle p(0),0\rangle\in A$. Let $(b_n)$ be a sequence of points in $K_i\cap B$ converging to $p$. Then $b_n':=\langle b_n(0),0\rangle\in B$ for each $n<\omega$ because $X$ is connectible.  The sequence $(b_n')$ converges to $p'$. Since $B$ is closed in $W$, $p'\in B$. We have $p'\in A\cap B$, a contradiction.

Suppose that $W[X]$ is not widely-connected. Let $A$ be a non-dense connected subset of $W$ with more than one point. Every non-degenerate proper subcontinuum of $K$ is an arc, therefore $\overline A$ is an arc. Let $e:[0,1]\hookrightarrow K$ be a homeomorphic embedding such that $e([0,1])=\overline A$. Let $a,b\in A$ with $a\neq b$, and let  $r,s\in [0,1]$ such that $e(r)=a$ and $e(s)=b$. Assume that $r<s$. Since $e^{-1}[A]$ is a connected subset of $[0,1]$ containing $r$ and $s$, we have $[r,s]\subseteq e^{-1}[A]$. Clearly $e([r,s])\not \subseteq \Delta$, so there exists $i<\omega$ such that $e([r,s])\cap K_i\neq\varnothing$.  Then $e^{-1}[K_i]\cap [r,s]$ is a non-empty open subset of $[r,s]$, so it contains an interval $I$. $e[I]$ is a non-degenerate connected subset of $W\cap K_i=\varphi_i[X]$, thus $X$ is not hereditarily disconnected.\end{proof}

\begin{ur} A widely-connected set cannot contain an interval of reals, so every widely-connected subset of $K$ is dense in $K$. \end{ur}

\section{Connectible sets $X_1$ and $X_2$}

Here we describe two connectible sets $X_1$ and $X_2$.  Both sets will be hereditarily disconnected and  dense in $C\times (0,1)$, so that $W[X_1]$ and $W[X_2]$ will be widely-connected.

\begin{ue}Let $C'$ be the set of all endpoints of intervals removed from $[0,1]$ in the process of constructing $C$, and let $C''=C\setminus C'$. Let $$X_1=\big(C'\times \mathbb Q\cap (0,1)\big) \cup \big(C''\times \mathbb P\cap(0,1)\big),$$ where $\mathbb Q$ and $\mathbb P$ are the rationals and irrationals, respectively.  Clearly $X_1$ is hereditarily disconnected and dense in $C\times (0,1)$. The reader may recognize $X_1$ as the Knaster-Kuratowski fan \cite{kk} minus its dispersion point.  Proving that $X_1$ is connectible is essentially the same as proving that the Knaster-Kuratowski fan is connected, so some details have been omitted from the proof below. \end{ue}

\begin{up}\label{lop}$X_1$ is connectible.\end{up}

\begin{proof}[Sketch of proof]  Let $A$ be a non-empty clopen subset of $X_1\cup C\times \{0\}$ and let $\langle c,r\rangle\in A$. Suppose for a contradiction that $\langle c,0\rangle\in B:=X_1\setminus A$. There are open sets $U\subseteq C$ and $V\subseteq (0,1)$ such that $\langle c,0\rangle\in U\times \{0\}\subseteq B$ and $\langle c,r\rangle\in X_1\cap (U\times V)\subseteq A$. Enumerate  $\mathbb Q \cap (0,1)=\{q_i:i<\omega\}$.  For each $i<\omega$ let \[C_i=\big\{c\in C:\langle c,q_i\rangle\in \overline A\cap \overline B\big\}.\]
Each $C_i$ closed and nowhere dense in $C$. By the Baire Category Theorem, $C\setminus (C'\cup \bigcup _{i<\omega} C_i)$ is dense, so there exists $d\in U\cap C''\setminus \bigcup _{i<\omega} C_i$.    $\overline A$ and $\overline B$ form a non-trivial clopen partition of the connected set $\{d\}\times[0,1)$, a contradiction.\end{proof}

\begin{ue2} Define $\|\cdot \|:\mathbb R ^\omega \to [0,\infty]$ by  \[\|x\|=\sqrt{\sum _{i=0} ^\infty x_i ^2}.\]  The Hilbert space $\ell ^2$ is the set $\{x\in \mathbb R ^\omega:\|x\|<\infty\}$ with the norm topology generated by  $\|\cdot\|$. The subspace $\mathfrak E:=\{x\in \ell^2:x_i\in \{0\}\cup \{1/n:n\in\mathbb N\} \text{ for each }i<\omega\}$ of $\ell ^2$  is called the \textit{complete Erd\H{o}s space} ($\mathfrak E$ is closed subset of the completely metrizable $\ell^2$, and is therefore complete). Let $f:[0,\infty)\to [0,1]$ be the function with the following graph.

\begin{figure}[H]
  \centering
  \includegraphics[scale=.5]{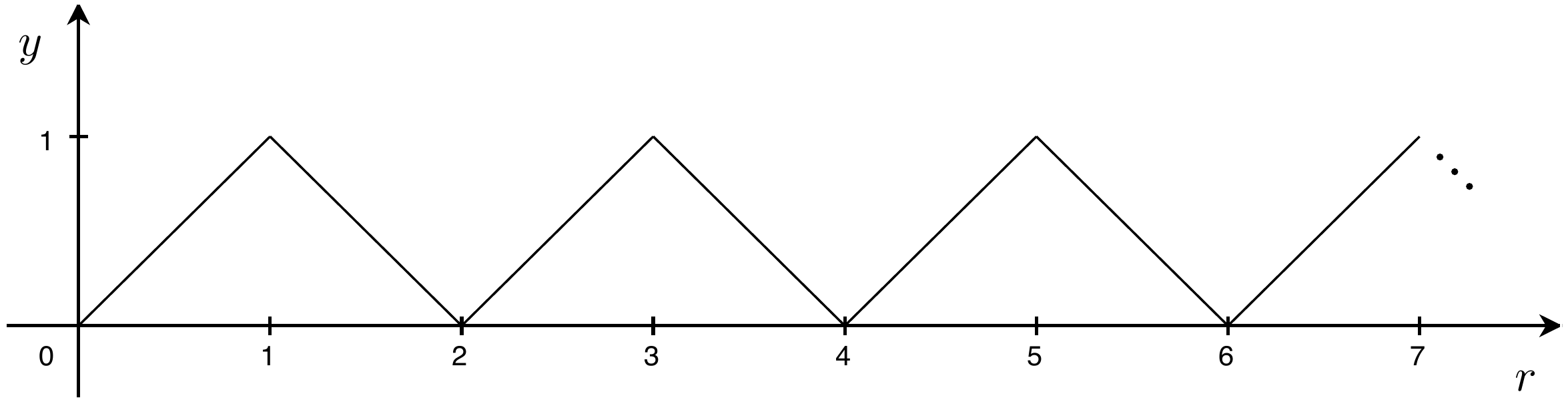} 
   \caption{$y=f(r)$}
\end{figure}

\noindent The Cantor set is the unique zero-dimensional compact metric space without isolated points, therefore  $ (\{0\}\cup \{1/n:n\in\mathbb N\})^\omega\simeq C$. We will think of these spaces as being equal. Define $$\xi: \mathfrak E\to C\times [0,1] \text{ by  }x\mapsto \langle x,f\|x\|\rangle.$$ Note that $\xi$ is continuous, as the inclusion $\mathfrak E \hookrightarrow  (\{0\}\cup \{1/n:n\in\mathbb N\})^\omega$ is continuous and $f$ and $\|\cdot\|$ are both continuous. Let \[X_2=\xi[\mathfrak E]\setminus (C\times \{0,1\}).\] Since $X_2$ has at most one point from each fiber $\{c\}\times (0,1)$, every two points in $X_2$ can be separated by a clopen partition of $X_2$. That is, $X_2$ is \textit{totally disconnected}, a stronger property than hereditarily disconnected. \end{ue2}

\begin{ur} If instead $f:[0,\infty)\to [0,1]$ is defined by $f(r)=1/(1+r)$, then $\xi$ is a homeomorphic embedding (see \cite{rob} and \cite{eng1} 6.3.24). By adding $C\times \{0\}$ to $\xi[\mathfrak E]$ and then contracting it to a point, one obtains a completely metrizable dispersion point space. However, in this case $\xi[\mathfrak E]$ is nowhere dense in $C\times (0,1)$. We will see that the sinusoidal $f$ makes $\xi$ a dense embedding, but destroys completeness.\end{ur}

\begin{up}$X_2$ is dense in $C\times (0,1)$.\end{up}

\begin{proof}Let $U\times (a,b)$ be a non-empty open subset of $C\times (0,1)$. Assume that $U$ is a basic open subset of $(\{0\}\cup \{1/n:n\in\mathbb N\})^\omega$; $U=\prod _{i<\omega} U_i$ and $n=\max\{i<\omega:U_i\neq X\}$.  For each $i\leq n$ choose $x_i\in U_i$.  Let $c\in (a,b)$. There exists $p\in  f^{-1} \{c\}$ such that $p^2>\sum_{i\leq n} x_i ^2$. Let $r=p^2-\sum_{i\leq n} x_i ^2$. There is an increasing sequence  of rationals $(q_i)$ such that $q_0=0$ and $q_i\to r$ as $i\to\infty$. Each $q_{i+1}-q_i$ is a positive rational number $\frac{a_i}{b_i}$ for some $a_i,b_i\in \mathbb N$. For each $i<\omega$ let $y^i\in\{\frac{1}{b_i}\}^{a_i\cdot b_i}$ be the finite sequence of $a_i\cdot b_i$  repeated entries $\frac{1}{b_i}$. Let $y={y^0} ^{\;\frown} {y^1} ^{\;\frown} {y^2} ^{\;\frown} ...$ be the sequence in  $\{1/n:n\in\mathbb N\}^\omega$ whose first $a_0\cdot b_0$ coordinates are $\frac{1}{b_0}$, whose next $a_1\cdot b_1$ coordinates are $\frac{1}{b_1}$, etc. Note that $\sum_{i=0}^\infty y_i^2=\sum_{i=0}^\infty q_{i+1}-q_i=r$. Now put $z:=(x\restriction n+1)^\frown y=\langle x_0,...,x_n,y_0,y_1,...\rangle$. We have $z\in U$ and $\|z\|=p$, so that  $\xi (z)=\langle z,c \rangle \in U\times (a,b)$.\end{proof}

The following is due to Erd\H{o}s \cite{erd}.

\begin{up}\label{bound}If $A$ is a non-empty clopen subset of $\mathfrak E$, then $\{\|x\|:x\in A\}$ is unbounded. \end{up}

\begin{proof}Let $A$ be a non-empty subset of $\mathfrak E$ such that $\{\|x\|:x\in A\}$ is bounded. We show $A$ has non-empty boundary. 

Let $N<\omega$ such that $\|x\|< N$ for each $x\in A$, and let $a^0\in A$. Define $a^1$ as follows. There is a least $j\in [1,N]$ such that $\langle 1,1,...,1,a^0_{j},a^0_{j+1},...\rangle \in \mathfrak E\setminus A$ (replacing $a^0_i$ with $1$ for each $i<j$). Let $a^1 _i =1$ if $i<j-1$ and $a^1 _i=a^0_i$ otherwise. Then $a^1\in A$ and $d(a^1,\mathfrak E \setminus A)\leq 1$. Let $j_0=0$ and $j_1=j$. 

Suppose $k>1$ and $a^n\in A$ and increasing integers $j_n$ have been defined, $n<k$, such that  $d(a^n,\mathfrak E \setminus A)\leq 1/n$ and $a^n _i =a^0_i$ for $i\geq j_n$. 
There exists $j'>j_{k-1}$ such that $a^0_i<1/k$ whenever $i\geq j '$. There is a least $j\in [1,kN]$ such that $$\langle a^{k-1} _0, ...,a^{k-1}_{j'-1}, 1/k,1/k,...,1/k, a^{k-1}_{j'+j},...\rangle \in \mathfrak E\setminus A.$$  Let $j_k=j'+j$. Define $a^k$ by letting $a^k _i=1/k$ if $j'-1<i<j_k-1$ and   $a^k _i=a^{k-1} _i$ otherwise. 

Finally, define $a$ by setting it equal to $a^k$ on $[j_{k-1},j_{k}]$, $k\in\mathbb N$.  $\|a\|\leq N$, otherwise there is a finite sum $\sum_{i=0} ^n a_i ^2$ greater than $N^2$, but then $\|a^k\|>N$ if $k>n$. Thus $d(a^0,a)\leq 2N$, and so $\sum_{i=n} ^\infty (a^0_i-a_i)^2 \to 0$ as $n\to\infty$. So $d(a^k,a)\to 0$ as $k\to\infty$. Therefore $a$ is a limit point of $A$. Also, $d(a,\hat a^k)\leq d(a, a^k)+d( a^k,\hat a ^k)$, where $\hat a^k$ is equal to $a^k$ with $a^k_{j_k-1}$ increased to $1/k$ (so $\hat a ^k \in \mathfrak E\setminus A$). By construction $d(a^k,\hat a^k)<1/k$, so it follows that $d(a,\hat a^k)\to 0$ as $k\to\infty$. Therefore $a$ is also a limit point of $\mathfrak E \setminus A$.  
\end{proof}

\begin{up}$X_2$ is connectible.\end{up}

\begin{proof}Suppose that $A$ is a clopen subset of $X_2 \cup (C\times \{0\})$, $\langle a,f\|a\| \rangle \in A$, and $\langle a,0\rangle\in \xi[\mathfrak E]\setminus A$. There is a clopen $U\subseteq C$ such that $$\langle a,0\rangle\in \xi[\mathfrak E]\cap (U\times \{0\})\subseteq \xi[\mathfrak E]\setminus A, $$ so that $\{\langle x,f\|x\|\rangle  :x\in U\text{ and } \|x\| \text{ is even}\}\subseteq \xi[\mathfrak E] \setminus A$. Let $n$ be an even integer greater than $\|a\|$. Then 
$$\xi^{-1}\big[A\cap U\times [0,1]\big]\cap \{x\in \mathfrak E:\|x\|<n\}=\xi^{-1}\big[A\cap U\times [0,1]\big]\cap \{x\in \mathfrak E:\|x\|\leq n\}$$
 is a non-empty (it contains $a$) clopen subset of $\mathfrak E$.  Its set of norms is bounded above (by $n$). This contradicts Proposition \ref{bound}.\end{proof}

\section{Answer to Mioduszewski's question}

Here we provide a negative answer to the question: \textit{Does every separable metric widely-connected $W$ have a metric compactification $\gamma W$ such that  every composant of $\gamma W$ intersects $W$ in a finite set?}.  This is Question 23(c)(2) in \cite{miodproblem}.

Recall that if $X$ is a continuum and $P\subseteq X$, then $P$ is a \textit{composant} of $X$ if there exists $p\in X$ such that $P$ is the union of all proper subcontinua of $X$ that contain $p$.\footnote{The composant of $K$ containing the point $\langle 0,0\rangle$  is a dense one-to-one continuous image of the half-line $[0,\infty)$. It is partially visible in Figure 2. Every other composant of $K$ (there are $\mathfrak c=2^\omega$ of them) is a dense one-to-one continuous image of the real line $(-\infty,\infty)$. } In the proof below, we also need Lavrentieff's Theorem: If $M$ and $N$ are completely metrizable spaces, $A\subseteq M$, $B\subseteq N$, and $h:A\to B$ is a homeomorphism, then $h$ extends to a homeomorphism between $G_\delta$-sets containing $A$ and $B$. More precisely, there is a $G_\delta$-in-$M$-set $\tilde A\supseteq A$, a $G_\delta$-in-$N$-set $\tilde B\supseteq B$, and a homeomorphism $\tilde h:\tilde A\to \tilde B$ such that $\tilde h \restriction A=h$.

\begin{ut}\label{hp}If $\gamma W$ is a metric compactification of $W:=W[X_1]$, then there is a composant $P$ of $\gamma W$ such that $\abs{W\cap P}=\mathfrak c$.\end{ut}

\begin{proof}Let $\varphi_0:C\times (0,1)\to K_0$ be the homeomorphism defined in Section 2. By Lavrentieff's Theorem, there is a $G_\delta$-in-$C\times (0,1)$-set $Z\supseteq X_1$ and a homeomorphic embedding $$h:=\widetilde{ \varphi_0\restriction X_1}:Z\hookrightarrow \gamma W$$ such that $h\restriction X_1=\varphi_0\restriction X_1$.  Let $\pi$ be the first coordinate projection in $C\times (0,1)$ and let $\{q_i:i<\omega\}$ enumerate $\mathbb Q \cap (0,1)$. Note that $$C''\cap \bigcap_{i<\omega}\pi [Z\cap C\times \{q_i\}]$$ is non-empty because it is a countable intersection of dense $G_\delta$'s in $C$ (see 1.4.C(c) in \cite{eng1}). So there exists $c\in C''$ such that $\{c\}\times [0,1]\subseteq Z$. If $P$ is the composant of $\gamma W$ containing the arc $h\big[\{c\}\times [0,1]\big]$, then $W\cap P$ contains a copy of the irrationals, so that  $\abs{W\cap P}=\mathfrak c$.\end{proof}

%This  provides a negative answer to Question 23(c)(2) in \cite{miodproblem}. 
 
\section{Necessary conditions for completeness}

\begin{up}\label{comp}$W[X]$ is completely metrizable if and only if the same is true of $X$.\end{up}

\begin{proof}Suppose that $X$ is completely metrizable. Then each $\varphi_i[X]$ is a $G_\delta$-set in $K$.  Further, the sets $\varphi_i [X]$ lie in pairwise disjoint open regions of $K$, and so their union is $G_\delta$ in $K$. Thus $W[X]$ is the union of two $G_\delta$'s, $\Delta$ and $\bigcup_{i<\omega}\varphi_i[X]$, which is again a $G_\delta$ (in $K$). The converse is true because $W[X]$ contains an open copy of $X$. \end{proof}

\begin{ur}$X_1$ is not complete because it contains closed copies of the rationals, e.g. in the fibers above $C'$. Therefore $W[X_1]$ is not complete. Moreover, the completion of $W[X_1]$ is not widely-connected by the proof of Theorem \ref{hp} (we showed that every completely metrizable supserset of $W[X_1]$ contains an arc).   \end{ur}

\begin{ut}\label{lop}If $G$ is a dense $G_\delta$ in $C\times \mathbb R$, then 
$$D:=\{c\in C:G\cap \{c\}\times \mathbb R\text{ is a dense $G_\delta$ in }\{c\}\times \mathbb R\}$$ is a dense $G_\delta$ in $C$.\footnote{This strengthens   \S1 in Section 1 of  \cite{gdgraph}.}\end{ut}

\begin{proof}Let $\{G_i:i<\omega\}$ be a collection of open sets in $C\times \mathbb R$ such that $G=\bigcap_{i<\omega} G_i$, and let $\{V_j:j<\omega\}$ be a countable basis for $\mathbb R$ consisting of non-empty sets. For each $i$ and $j$, define $$F(i,j)=\{c\in C:G_i\cap \{c\}\times  V_j=\varnothing\}.$$ 

$D=C\setminus \bigcup _{i,j<\omega} F(i,j)$:  Suppose that $d\in D$ and $i,j<\omega$.  Then $G\cap \{d\}\times V_j\neq\varnothing$ by density of $G\cap \{d\}\times \mathbb R$ in $\{d\}\times \mathbb R$. As $G\subseteq G_i$, we have  $G_i\cap \{d\}\times V_j\neq\varnothing$, so that $d\notin F(i,j)$.  Thus $D\subseteq C\setminus \bigcup _{i,j<\omega} F(i,j)$. Now let $c\in C\setminus \bigcup _{i,j<\omega} F(i,j)$. Fix $i<\omega$. As $G_i\cap\{c\}\times V_j\neq\varnothing$ for each $j<\omega$, we have that $G_i\cap \{c\}\times \mathbb R$ is dense in $\{c\}\times \mathbb R$. Thus $G\cap \{c\}\times \mathbb R=\bigcap_{i<\omega} G_i\cap \{c\}\times \mathbb R$ is a countable intersection of dense open subsets of $\{c\}\times \mathbb R$. By the Baire property of $\mathbb R$, $G\cap \{c\}\times \mathbb R$ is a dense $G_\delta$ in $\{c\}\times \mathbb R$, whence $c\in D$.

Each $F(i,j)$ is closed and nowhere dense in $C$: Fix $i,j<\omega$ and let $F=F(i,j)$. $F$ is closed in $C$: Let $c\in C\setminus F$.  There exists $r\in \mathbb R$ such that $\langle c,r\rangle\in G_i\cap \{c\}\times V_j$. There is an open set $U\times V\subseteq G_i$ with $c\in U$ and  $r\in V\subseteq V_j$. Then $F\cap U=\varnothing$. So $C\setminus F$ is open. $F$ is nowhere dense in $C$: If $U\subseteq C$ is non-empty and open, then by density of $G$ there exists $\langle c,r \rangle\in G\cap U\times V_j$. Then $c$ witnesses that $U\not \subseteq F$.

It now follows that from the Baire property of $C$ that $D$ is a dense $G_\delta$ in $C$.\end{proof}

\begin{ur}By Theorem \ref{lop} and the remark following Proposition \ref{one}, every completely metrizable widely-connected subset of $K$ has uncountable intersection with some composant of $K$. The original widely-connected set by Swingle \cite{swi} has exactly one point from each composant of $K$.  Miller's biconnected set \cite{mil} has  countable intersection with each composant, and the same is true of $W[X_2]$.  None of these examples are complete. \end{ur}

\begin{ut}\label{ttt}If $W$ is a  completely metrizable widely-connected subset of $K$, then there is a closed $F\subseteq K$ such that   $W\cap F=\varnothing$, $K\setminus F$ is connected,  and $F$ intersects every composant of  $K$.\end{ut}

\begin{proof}Since $W$ is completely metrizable, there is a collection $\{F_n:n<\omega\}$  of closed subsets of $K$  such that $W=K\setminus \bigcup _{n<\omega} F_n$. Let $\varphi_0$ and $K_0$ be as defined in Section 2.  By the widely-connected property of $W$, for each $c\in C$ there exists $n<\omega$ such that $F_n\cap \varphi_0[\{c\}\times (0,1)]\neq\varnothing$. Thus $C=\bigcup _{n<\omega} \pi [\varphi_0^{-1}F_n]$ where $\pi$ is the first coordinate projection in $C\times (0,1)$.  By the Baire Category Theorem there exists $N<\omega$ and a non-empty open $U\subseteq C$ such that $U\subseteq \pi [\varphi_0^{-1}F_N]$. The set $F:=F_N$ is as desired.  Clearly $W\cap F=\varnothing$.   And  $K\setminus F$ is connected since $W$ is connected and dense in $K$ (see the remark following the proof Proposition \ref{one}). Finally, the open set $\varphi_0[U\times (0,1)]$ meets every composant of $K$, so $F$ does as well.\footnote{Theorem 3 of \cite{cooks} says, more generally, that if $(F_n)$ is a sequence of closed subsets of an indecomposable continuum $I$, and each composant of $I$ intersects $\bigcup_{n<\omega}F_n$, then some $F_N$ intersects each composant of $I$.} \end{proof}

\section{Solution to the Erd\H{o}s-Cook problem}\label{ec}

The main result of \cite{deb} is that $K$ has a closed subset $F$ with the properties stated in Theorem \ref{ttt}. Considered as a subset of $C\times [0,1]$, $F$ is the closure of the graph of a certain hereditarily discontinuous function. Specifically, let $D=\{d_n:n<\omega\}$ be a  dense subset of the non-endpoints in $C$, let $(a_n)$ be a sequence of positive real numbers such that $\sum _{n=0}^\infty a_n=1$, and define $f:C\to [0,1]$ by $$f(c)\hspace{2mm}=\hspace{-2mm}\sum_{\{n<\omega:d_n<c\}} \hspace{-2mm}a_n.$$ Then $F=\overline{\text{Gr}(f)}$.

\begin{ue3}Enumerate the rationals $\mathbb Q=\{q_n:n<\omega\}$. Let $$Y=C\times \mathbb R \setminus \bigcup _{n<\omega} \overline{\text{Gr}(f+q_n)}.$$  Define $\Xi:C\times \mathbb R \to C\times (0,1)$  by $\Xi(\langle c,r\rangle)=\big\langle c,\frac{\arctan(r)}{\pi}+\frac{1}{2}\big\rangle.$ Let \[X_3= \Xi [Y].\] 
\end{ue3}

Recall that if $X$ is a space then $Q\subseteq X$ is a \textit{quasi-component} of $X$ if $Q=\bigcap \{A:A\text{ is a clopen subset of }X\text{ and }q\in A\}$ for each $q\in Q$.

\begin{up}\label{line}If $U$ is open in $C$ and $V=(a,b)$ is an open interval in $\mathbb R$, then for every $c\in U$,  $(\{c\}\times V) \setminus \overline{\Gr(f)}$ is a quasi-component of $(U\times V) \setminus \overline{\Gr(f)}$.\end{up}

\begin{proof}We first show that if $d_n\in  U\cap D$ ($n<\omega$), then $(\{d_n\}\times V) \setminus \overline{\text{Gr}(f)}$ is contained in a quasi-component of $(U\times V) \setminus \overline{\text{Gr}(f)}$. Fix $n<\omega$ such that $d_n\in U$. Let $r=f(d_n)$ and  $s=f(d_n)+a_n (=\inf\{f(c):d_n<c\})$. Since $f$ is a non-decreasing function with set of (jump) discontinuities $D$, $F:=\overline{\text{Gr}(f)}$ meets a vertical interval $\{c\}\times [0,1]$ in exactly one point $\langle c,f(c)\rangle$ if $c\in C\setminus D$, and meets  $\{d_n\}\times [0,1]$ in exactly two points $\langle d_n,r\rangle$ and $\langle d_n,s\rangle $. Thus $(\{d_n\}\times V) \setminus F$ is equal to
\begin{align*} 
&\{d_n\}\times (a,b), \\
&\{d_n\}\times (a,r)\cup(r,b), \\
&\{d_n\}\times (a,s)\cup (s,b)\text{, or} \\
&\{d_n\}\times (a,r)\cup (r,s)\cup (s,b),
\end{align*}
depending on whether neither, one, or both of $r$ and $s$ are in $V$. Because $d_n$ is a non-endpoint of $C$, it is the limit of  increasing and decreasing sequences of points in $C\setminus D$. Suppose that $r\in V$. Then the intervals $\{c\}\times (a,\min\{b,f(c)\})$, $d_n<c\in U$, are well-defined subsets of $(U\times V)\setminus F$. They limit to points vertically above and below $\langle d_n,r \rangle$, so that $\{d_n\}\times (a,r)\cup(r,\min\{b,s\})$ is contained in a quasi-component of $(U\times V)\setminus F$. Similarly, if $s\in V$ then  in $(U\times V)\setminus F$  the intervals $\{c\}\times (\max\{a,f(c)\},b)$, $d_n>c\in U\setminus D$, bridge the gap in $\{d_n\}\times (\max\{a,r\},s)\cup(s,b)$. Thus $(\{d_n\}\times V)\setminus F$ is contained in a quasi-component of $(U\times V)\setminus F$. 

Since $D$ is dense in $U$ and $(\{d_n\}\times V)\setminus F$ is dense in $\{d_n\}\times V$ for each $n<\omega$,  $(\{c\}\times V)\setminus F$ is contained in a quasi-component of $(U\times V)\setminus F$ for each $c\in U$. The opposite inclusion holds because $C$ has a basis of clopen sets.\end{proof}

\begin{figure}[H]
  \centering
  \includegraphics[scale=.6]{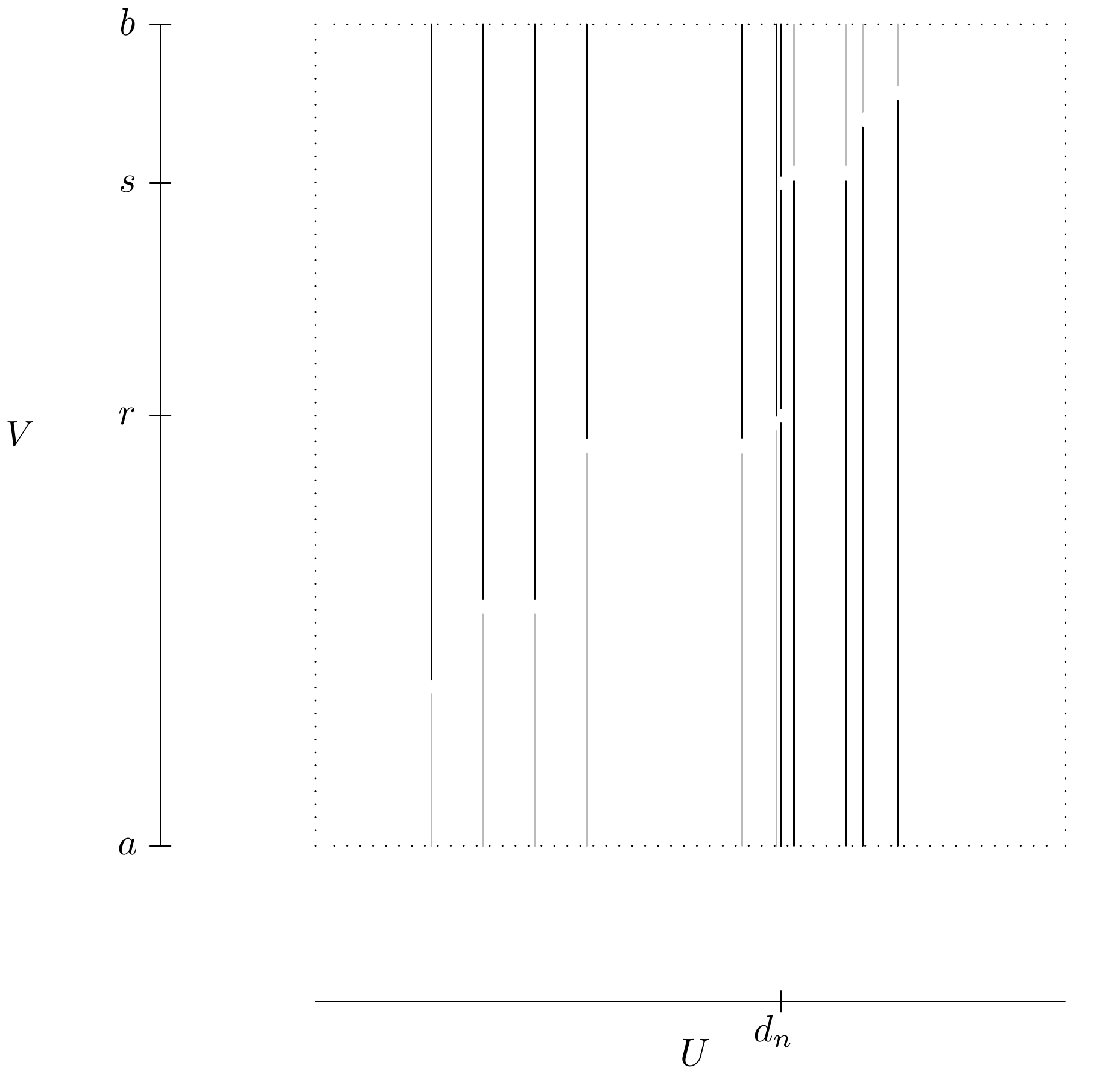} 
   \caption{$(U\times V)\setminus \overline{\text{Gr}(f)}$ when $d_n\in U$ and $r,s\in V$}
\end{figure}

\begin{up}\label{ll}For each $c\in C$, $ (\{c\}\times \mathbb R)\cap Y$ is a quasi-component of $Y$.\end{up}

\begin{proof}We need to show each fiber $ (\{c\}\times \mathbb R)\cap Y$ is contained in a quasi-component of $Y$. Suppose not. Then there exists $c\in C$ and a partition of $Y$ into disjoint clopen sets $A$ and $B$ such that  $A\cap(\{c\}\times \mathbb R)\neq\varnothing$ and $B\cap(\{c\}\times \mathbb R)\neq\varnothing$. Density of $Y$ in $C\times \mathbb R$ implies that $\overline A \cup \overline B =C\times \mathbb R$. In particular, $\{c\}\times \mathbb R\subseteq \overline A \cup \overline B$, so  $\overline A \cap \overline B\neq\varnothing$ by connectedness of $\{c\}\times \mathbb R$.  Now apply the Baire Category Theorem in $\overline A \cap \overline B$. As $\overline A \cap \overline B\subseteq \bigcup _{n<\omega} \overline{\text{Gr}(f+q_n)}$, there exists $N<\omega$, an open $U\subseteq C$, and an open interval $V\subseteq \mathbb R$ such that $$\varnothing \neq  \overline A \cap \overline B \cap (U\times V)\subseteq F_N:= \overline{\text{Gr}(f+q_N)}.$$  

Notice that
\begin{align*} 
\overline A \cap (U\times V)\setminus F_N=((C\times \mathbb R) \setminus \overline B)\cap (U\times V)\setminus F_N&\text{; and} \\
\overline B \cap (U\times V)\setminus F_N=((C\times \mathbb R) \setminus \overline A)\cap (U\times V)\setminus F_N&.
\end{align*} 
So $\overline A \cap (U\times V)\setminus F_N$ and $\overline B \cap (U\times V)\setminus F_N$ are open sets.  They are  disjoint and their union is equal to $(U\times V)\setminus F_N$.  By Proposition \ref{line},   $(\{c\}\times V) \setminus F_N$ is contained in either $\overline A$ or $\overline B$ whenever $c\in U$. So 
$$U_1:=\pi\big[\overline A \cap (U\times V)\setminus F_N\big]\text{ and } U_2:=\pi\big[\overline B \cap (U\times V)\setminus F_N\big]$$
are disjoint open subsets of $C$, $\pi$ being the first coordinate projection ($\pi$ is an open mapping).  Further, $U_1\cup U_2=U$ because $\abs{F_N\cap (\{c\}\times V)}\leq 2$ for each $c\in C$. Hence $U_1\times V$ and $U_2\times V$ form a clopen partition of $U\times V$. As $A\cap (U\times V)\subseteq U_1\times V$ and $B\cap (U\times V)\subseteq U_2\times V$, we have $\overline A \cap \overline B \cap (U\times V)=\varnothing$, a contradiction.\end{proof}

\begin{up}\label{oo}$X_3$ is connectible. \end{up}

\begin{proof}Suppose that $A$ is a clopen subset of $X_3\cup (C\times \{0\})$, $c\in C$, and $A\cap (\{c\}\times (0,1))\neq\varnothing$. The homeomorphism $\Xi$ preserves the form of quasi-components. Therefore,  Proposition \ref{ll} implies that $X_3\cap (\{c\}\times (0,1))\subseteq A$. Since $X_3$ is dense in $\{c\}\times [0,1)$ and $A$ is closed, we have $\langle c,0\rangle\in A$. \end{proof}

\begin{ut}$W[X_3]$ is widely-connected and completely metrizable.\end{ut}

\begin{proof}Each fiber $(\{c\}\times \mathbb R)\cap Y$ is a real line minus one or two shifted copies of $\mathbb Q$, and is therefore hereditarily disconnected and dense in $\{c\}\times \mathbb R$.  It follows that  $X_3$ is hereditarily disconnected and dense in $C\times  (0,1)$. By Propositions \ref{one} and \ref{oo}, $W[X_3]$ is widely-connected.  $X_3$ is a $G_\delta$-subset of $C\times (0,1)$ by design, so it is completely metrizable.  By Proposition  \ref{comp}, $W[X_3]$ is  completely metrizable.\end{proof}

E.W. Miller \cite{mil} and M.E. Rudin \cite{rud1} showed that, consistently, a widely-connected subset of the plane can be biconnected. This is not the case with $W[X_3]$.

\begin{up}\label{10}$W[X_3]$ is not biconnected.\end{up}

\begin{proof}We first show that if $S$ is a countable  subset of $W:=W[X_3]$, then $W\setminus S$ is  connected. To that end, let $S=\{x_n:n<\omega\}$ be a countable subset of $W$.  Note that Proposition \ref{line} holds if $\overline{\text{Gr}(f)}$ is replaced by $\overline{\text{Gr}(f)}\cup \{x\}$ for any $x\in C\times \mathbb R$. By the proofs of Propositions \ref{ll} and \ref{oo}, it follows that for each $i<\omega$,  $$K_i \setminus \Big(\bigcup _{n<\omega} \big(\varphi_i \circ \Xi \big[\overline {\text{Gr}(f+q_n)}\big]\big) \cup\{x_n\}\Big)$$ is connectible (considered as a subset of $C\times (0,1)$). By Proposition \ref{one}, $W':=\Delta \cup (W\setminus S)$ is connected.  

Let $W''=W\setminus S$. We want to show that $W''$ is connected.  Suppose to the contrary that $\{A,B\}$ is a non-trivial clopen partition of $W''$. Then $\text{cl}_{W'} A \cap \text{cl}_{W'} B$ is a non-empty subset of $\Delta$. Let $p\in \text{cl}_{W'} A \cap \text{cl}_{W'} B$.  There exists $j<\omega$ such that $p\in \overline {K _j}$. There is a $K$-neighborhood of $p$ that identifies with $C\times [-1,1]$ in such a way that $C\times [-1,0]\subseteq \overline{K_0}$, $C\times [0,1]\subseteq \overline{K_j}$, and $p\in C\times \{0\}$. We think  of $C\times [-1,1]$ as being an actual subset of $K$, and give local coordinates $\langle d,0\rangle$ to $p$. 

Note that $W''\cap\{c\}\times [-1,1]$ is dense in $\{c\}\times [-1,1]$ for each $c\in C$. Since $W''\cap K_0$ and $W''\cap K_j$ are connectible, for each point $\langle c,0\rangle\in A$  we have  $W''\cap (\{c\}\times [-1,1])\subseteq A$, and similarly for $B$.  If $p$ is the limit of sequences of points in $A$ and $B$ that are in $C\times \{0\}$, then points in $W''\cap (\{d\}\times (0,1])$ would be limit points of fibers in $A$ and $B$. This cannot happen.  Therefore there is an open $U\subseteq C$ such that $p\in  W''\cap (U\times \{0\})\subseteq A$, without loss of generality. Since $p\in \overline B$, there exists $b\in B$ such that $b(0)$, the first local coordinate of $b$, is in $U$. There is an open set $T\times V$ with $b\in W''\cap (T\times V)\subseteq B$. Since $W''\cap \Delta$ is dense in $\Delta$, there exists $a\in A\cap [(U\cap T)\times \{0\}]$. We have a contradiction:  $W'\cap \{a(0)\}\times [-1,1]$ meets both $A$ and $B$. We have shown that $W''$ is connected.

Thus every separator of $W$ is uncountable. Let $$\mathscr S=\{S:S \text{ is a closed subset of }W\text{ and }W\setminus S \text{ is not connected}\}.$$ If $Z$ is a Bernstein set in $W$, then  $Z\cap S\neq\varnothing$ and $(W\setminus Z)\cap S\neq\varnothing$ for each $S\in \mathscr S$, implying that both $Z$ and $W\setminus Z$ are connected.   So $W$ is not biconnected.

 A set like $Z$ can be constructed as follows using separability and completeness of $W$. Let $\{S_\alpha:\alpha<\mathfrak c\}$ be an enumeration  of $\mathscr S$. Every member of $\mathscr S$ is closed and uncountable, and thus has  cardinality $\mathfrak c$. Let $y_0$ and $z_0$ be distinct points in $S_0$.  If $\alpha<\mathfrak c$ and $y_\beta$ and $z_\beta$ have been defined for $\beta<\alpha$, then there are two distinct points $y_\alpha,z_\alpha\in S_\alpha \setminus (\{y_\beta:\beta<\alpha\}\cup \{z_\beta:\beta<\alpha\})$. Then $Z=\{z_\alpha:\alpha<\mathfrak c\}$ is as desired. \end{proof}

\section{Related question}

%We would like to know:

Due to Bernstein sets, every separable completely metrizable biconnected space has a countable separator.

\begin{uq}Does every separable completely metrizable biconnected space have a dispersion point?\end{uq}

We have already seen that Miller's biconnected set  is not completely metrizable. Rudin's examples in \cite{rud1} and \cite {rud2} fail to be complete for similar reasons.  These are currently the only  known metric examples without dispersion points.

\setstretch{1}

\end{document}